\documentclass[11pt]{amsart}
\textheight
9in
\usepackage{amsmath,amsthm,amsfonts,amssymb,latexsym,epsfig}

\newtheorem{theorem}{Theorem}[section]

\newtheorem{lemma}{Lemma}[section]

\newtheorem{cor}{Corollary}[section]

\numberwithin{equation}{section}

\theoremstyle{definition}

\theoremstyle{remark}

\begin{document}
\title{A complete monotonicity result involving the $q$-polygamma functions}
\author{Peng Gao}
\address{Department of Mathematics, School of Mathematics and System Sciences, Beijing University of Aeronautics and Astronautics, P. R. China}
\email{penggao@buaa.edu.cn}
\thanks{The author is supported in part by NSFC grant 11371043 and the Fundamental Research Funds for the Central Universities.}
\begin{abstract}
We present some completely monotonic functions involving the
$q$-polygamma functions, our result generalizes some known results.
\end{abstract}

\maketitle
\section{Introduction}
\label{sec 1} \setcounter{equation}{0}
   For a positive real number $x$ and $q \neq 1$, the $q$-gamma function is given by
\begin{eqnarray*}
  \Gamma_q(x)&=& \left\{\begin{array}{ll}
      \displaystyle(1-q)^{1-x}\prod^{\infty}_{n=0}\frac {1-q^{n+1}}{1-q^{n+x}}, & 0<q<1; \\
      \displaystyle(q-1)^{1-x}q^{\frac
1{2}x(x-1)}\prod^{\infty}_{n=0}\frac {1-q^{-(n+1)}}{1-q^{-(n+x)}}, & q>1.
\end{array}\right.
\end{eqnarray*}

     Note that \cite[(1.4)]{A&G} the limit of $\Gamma_q(x)$ as $q \rightarrow
    1$ yields the well-known Euler's gamma function:
\begin{equation*}
    \lim_{q \rightarrow
    1}\Gamma_q(x)=\Gamma(x)=\int^{\infty}_0 t^x e^{-t}\frac
    {dt}{t}.
\end{equation*}

    Recall that a function $f(x)$ is said to
be completely monotonic on $(a, b)$ if it has derivatives of all
orders and $(-1)^kf^{(k)}(x) \geq 0, x \in (a, b), k \geq 0$.  There exists an extensive and rich literature on inequalities for the gamma and $q$-gamma functions of positive real numbers.
Many of these inequalities follow from the
monotonicity properties of functions which are closely related to
$\Gamma$ (resp. $\Gamma_q$) and its logarithmic derivative $\psi$
(resp. $\psi_q$) as $\psi'$ and $\psi'_q$ are completely
    monotonic functions on $(0, +\infty)$. The derivatives $\psi'_q,
\psi''_q, \ldots$ are called the $q$-polygamma functions.

    For positive integers $r, m, n, s$, we denote
\begin{align}
\label{20}
  \alpha_{r,m,n,s} =
     \frac {(m-1)!(n-1)!}{(r-1)!(s-1)!};
     \ \alpha_{r,m,n,0} =  \frac {(m-1)!(n-1)!}{(r-1)!}; \
\beta_{r,m,n,s} =
     \frac {m!n!}{r!s!}.
\end{align}

    For integers $r \geq m \geq n \geq  s\geq 0$ and any real number $t$, we
   define
\begin{align}
\label{02}
   F_{r, m,n,s}(x; t)= (-1)^{m+n}\psi^{(m)}(x)\psi^{(n)}(x) -
   t(-1)^{r+s}\psi^{(r)}(x)\psi^{(s)}(x),
\end{align}
   where we set $\psi^{(0)}(x)=-1$ for convenience.

   In \cite[Theorem 4.1]{Gao1}, it is shown that when $m+n=r+s$, the
   function
   $F_{r, m,n,s}(x; \alpha_{r,m,n,s})$ is completely monotonic on $(0, +\infty)$, while $-F_{r, m,n,s}(x; \beta_{r, m,n, s})$
   is also completely monotonic on $(0, +\infty)$ when $s>0$.  This gives a generalization of a result of Alzer and Wells \cite[Theorem 2.1]{A&W},  which
   asserts that for $n \geq 2$,
   the function $F_{n+1, n, n, n-1}(x; t)$ is strictly completely monotonic on $(0, +\infty)$ if and only if $t \leq
   (n-1)/n$ and $-F_{n+1, n, n, n-1}(x; t)$ is strictly completely monotonic on $(0, +\infty)$ if and only if $t \geq
   n/(n+1)$. Following the methods in the proof of \cite[Theorem 2.1]{A&W}, it is easy to show that the numbers $\alpha_{r,m,n,s}, \beta_{r, m,n, s}$ are best
   possible (see also \cite{Gao3}).

   A special case of \cite[Theorem 4.1]{Gao1} with $m=n=1, r=2,s=0$ implies that the function $
  (\psi'(x))^2+\psi''(x)$
    is completely monotonic on $(0, \infty)$. In particular, it implies that
\begin{align}
\label{1.2}
  (\psi'(x))^2+\psi''(x) \geq 0, \quad x>0.
\end{align}
   a result established in the
   proof of \cite[(4.39)]{alz2}(with strict inequality).

   We may regard the gamma function as a $q$-gamma function with $q=1$ as $\lim_{q \rightarrow 1}\psi_q(x)=\psi(x)$ (see
     \cite{K&S}). In this manner, many completely monotonic functions involving $\Gamma_q(x)$ and $\psi_q(x)$
      are inspired by their analogues involving $\Gamma(x)$ and $\psi(x)$.

   When $q>1$, the $q$-analogue of inequality \eqref{1.2} is given in \cite[Lemma 4.6]{A&G}:
\begin{align}
\label{1.3}
  (\psi'_q(x))^2+\psi''_q(x)>0, \quad x>0.
\end{align}
    In \cite[Theorem 1.1]{Qi}, it is shown that the function given in \eqref{1.3} is completely
monotonic for $q > 1$ on $(0, \infty)$.

    In \cite[Theorem2.1]{Ba1}, it is further shown that the function
\begin{align}
\label{1.4}
  (\psi'_q(x))^2+\psi''_q(x)-\ln q \cdot \psi'_q(x)
\end{align}
   is completely monotonic for $q>0$ on $(0, \infty)$.

    We note that inequality \eqref{1.2} follows from the limiting case $c \rightarrow 0^+$ of the following
inequalities established in the proof of \cite[Theorem
   1.1]{C}:
\begin{align}
\label{1.5}
   \frac 1{c}\Big(\psi(x+c)-\psi(x) \Big )^2 > \psi'(x)-\psi'(x+c), \quad 0 < c < 1, x > 0.
\end{align}

   More generally, one may replace the derivatives in the expression of $F_{r, m,n,s}(x; t)$ by finite differences and study the complete monotonicity property of the resulting function and this is done in \cite{Gao3}.

   A $q$-analogue for inequality \eqref{1.5} is given in \cite[Theorem 4.2]{Gao1}, which asserts that for fixed $0<q<1$ and $0<c<1$,
\begin{align*}
  \frac {1-q}{1-q^c}\Big(\psi_q(x+c)-\psi_q(x) \Big )^2 > q^x(\psi'_q(x)-\psi'_q(x+c))> \Big(\psi_q(x+c)-\psi_q(x) \Big )^2, \quad x>0,
\end{align*}
  with the above inequalities reversed when $c>1$.

  Motivated by the discussions above, it is natural to replace the derivatives in \eqref{1.4} by finite differences to study the complete monotonicity property of the modified function. It is our goal in this paper to study a more general case, the complete monotonicity property of a $q$-analogue of the function defined in \eqref{02}.

   For a given function $f(x)$, any real number $c \neq 0$, we denote
\begin{align*}
   \Delta f(x;c)=\frac {f(x+c)-f(x)}{c}.
\end{align*}

   For integers $r \geq m \geq n \geq  s\geq 0$, real numbers $q>0, q \neq 1, c \neq 0$, we
   define
\begin{align*}
    & F_{r, m,n,s}(x; q, c) \\
    = & (-1)^{m+n}\Delta\psi^{(m-1)}_q(x;c)\Delta\psi^{(n-1)}_q(x;c) -
   \alpha_{r,m,n,s}(-1)^{r+s}\Delta\psi^{(r-1)}_q(x;c)\Delta\psi^{(s-1)}_q(x;c),  \notag
\end{align*}
   where $\alpha_{r,m,n,s}$ is given in \eqref{20} and we set $\psi^{(0)}_q(x)=\psi_q(x), \psi^{(-1)}_q(x)=-x$ for convenience.

    We further define for integers $m \geq 1$, real numbers $q>0, q \neq 1, c \neq 0$,
\begin{align*}
     G_{m}(x;q, c) =mF_{m+1, m,1,0}(x; q, c)
    - (-1)^{m+1}d_{m,q}\ln q \cdot \Delta\psi^{(m-1)}_q(x;c),
\end{align*}
   where
\begin{eqnarray}
\label{1.7}
  d_{m,q}&=& \left\{\begin{array}{ll}
      \displaystyle m-1, & 0<q<1, m \geq 2, \\
      \displaystyle 1, & q>1 \quad \mbox{or} \quad 0<q<1, m=1.
\end{array}\right.
\end{eqnarray}

  Our results are the following:
\begin{theorem}
\label{thm1}
  Let $q>0, q \neq 1$ and $0<c<1$ be fixed. Let $m$ be any fixed positive integer. The function $G_{m}(x;q,c)$
  is completely monotonic on $(0, \infty)$. Moreover, when $m=1, 2$, the function $-G_{m}(x;q,c)$ is completely monotonic on $(0, \infty)$ when $c>1$.
\end{theorem}

\begin{theorem}
\label{thm2}
  Let $q>0, q \neq 1$ and $0<c<1$ be fixed. Let $r \geq m \geq n \geq s$ be positive integers satisfying $r+s=m+n$.
  When $s=1$, the function $F_{r, m,n,1}(x; q, c)$
  is completely monotonic on $(0, \infty)$ for $0<q<1$. When $s=2$, the function $F_{r, m,n,2}(x; q, c)$
  is completely monotonic on $(0, \infty)$ for $q>0, q \neq 1$. When $s=3$, the function $F_{r, m,4,3}(x; q, c)$
  is completely monotonic on $(0, \infty)$ for $q>0, q \neq 1$.
\end{theorem}

   The $c \rightarrow 0^+$ analogues of the results in Theorem \ref{thm1} and \ref{thm2} can be similarly established using the methods in the proofs of these theorems. Therefore we only state these results below and leave their proofs to the reader.
\begin{cor}
\label{cor1}
  Let $q>0, q \neq 1$ be fixed. Let $m$ be any fixed positive integer. Then the function $m(-1)^{m+1}\psi'_q(x)\psi^{(m)}_q(x) +
    (-1)^{m+1}\psi^{(m+1)}_q(x)- (-1)^{m+1}d_{m,q}\ln q\cdot \psi^{(m)}_q(x)$
  is completely monotonic on $(0, \infty)$, where $d_{m,q}$ is given in \eqref{1.7}.
\end{cor}

\begin{cor}
\label{cor2}
  Let $q>0, q \neq 1$ be fixed. Let $r \geq m \geq n \geq s \geq 0$ be positive integers satisfying $r+s=m+n$.  Let $\alpha_{r,m,n,s}$ be given in \eqref{20}. The function $(-1)^{m+n}\psi^{(m)}_q(x)\psi^{(n)}_q(x) -
   \alpha_{r,m,n,s}(-1)^{r+s}\psi^{(r)}_q(x)\psi^{(s)}_q(x)$
  is completely monotonic on $(0, \infty)$ when $0<q<1, s=1$ or $q>0, s=2$ or $q>0, s=3, n=4$.
\end{cor}
\section{Lemmas}
\label{sec 2.0} \setcounter{equation}{0}

     For integers $r \geq m \geq n \geq s \geq 1, r+s=m+n$, we define two sequences ${\bf a}(m,n),  {\bf b}(r,m,n,s)$ such that for $k \geq 1$,
\begin{align*}
    {\bf a}(m,n)_{k}& =\sum^{k}_{j=1}j^{m-1}(k-j)^{n-1}, \quad {\bf b}(r,m,n,s)_{k}=
 \alpha_{r,m,n,s}\sum^{k}_{j=1}(k-j)^{r-1}j^{s-1},
\end{align*}
   where $\alpha_{r,m,n,s}$ is defined in \eqref{20} and we define $0^0=1$.

   Our lemmas are concerned with the following inequality for $k \geq 1$:
\begin{align}
\label{2.1}
    {\bf a}(m,n)_{k} \geq {\bf b}(r,m,n,s)_{k}.
\end{align}

\begin{lemma}
\label{lem1}
   Inequality \eqref{2.1} holds for integers $k \geq 1,  r \geq m \geq n \geq 1, r+1=m+n$.
\end{lemma}
\begin{proof}
   For a given sequence ${\bf x}=(x_1, x_2, \cdots )$, we define $\Delta {\bf x}$ to be the sequence satisfying $(\Delta {\bf x})_{k}=x_{k+1}-x_k, k \geq 1$ and we define $\Delta^{(i)}{\bf x}=\Delta (\Delta^{(i-1)} {\bf x})$ for all $i\geq 2$.

   Using slightly different notations by taking integers $s \geq 0, r \geq s, t \geq s$, we see that inequality \eqref{2.1} corresponds to the following inequality for $s=0$:
\begin{align}
\label{2.200}
   {\bf a}(r+1,t+1)_{k} \geq {\bf b}(r+t-s+1,r+1,t+1,s+1)_{k}, \quad k \geq 1.
\end{align}

    It is ready to see that inequality \eqref{2.200} is a consequence of the following inequalities:
\begin{align}
\label{2.03}
    & (\Delta^{(s+1)}{\bf a}(r+1,t+1))_{k} \geq (\Delta^{(s+1)}{\bf b}(r+t-s+1,r+1,t+1,s+1))_{k}, \quad k \geq 1. \\
\label{2.04}
    & (\Delta^{(i)}{\bf a}(r+1,t+1))_{1} \geq (\Delta^{(i)}{\bf b}(r+t-s+1,r+1,t+1,s+1))_{1}, \quad 1 \leq i \leq s.
\end{align}

   When $s=0$, inequality \eqref{2.04} holds trivially so it remains to
   prove inequality \eqref{2.03}. We may now assume $r>1, t \geq 1$. When $t=1$,  inequality \eqref{2.03} becomes the following easily verified inequality:
\begin{align*}
    \sum^{k}_{j=1}j^{r} \geq \frac 1{r+1}k^{r+1}.
\end{align*}
    We then deduce that inequality \eqref{2.1} is valid when $s=0, n=2$.

    Now, we use induction on $k \geq 1$ to prove inequality \eqref{2.03} for all $r>1, t \geq 1$. The case $k=1$ is easy verified. Now assuming that inequality \eqref{2.03} is
valid for all $k \geq 1, 1 \leq t \leq T-1$. When $t = T$, we apply the binomial expansion to see that inequality \eqref{2.03} becomes
\begin{align}
\label{2.005}
    \sum^{T-1}_{i=0}\binom {T}{i} \sum^{k}_{j=1}j^{r}(k-j)^{i} \geq \frac {r!T!}{(r+T)!}k^{r+T}.
\end{align}
   Applying the induction assumption, we see that
\begin{align*}
    \binom {T}{i} \sum^{k}_{j=1}j^{r}(k-j)^{i} \geq  \binom {T}{i}\cdot \frac {r!i!}{(r+i)!}\sum^{k-1}_{j=1}j^{r+i}, \quad 1 \leq i \leq T-1.
\end{align*}
    Using this in \eqref{2.005}, we see that it remains to show that (with empty sums being $0$)
\begin{align*}
    S_k:= \sum^{T-1}_{i=1}\binom {r+T}{r+i}\sum^{k-1}_{j=1}j^{r+i}+\binom {r+T} {r}\sum^{k}_{j=1}j^{r} \geq k^{r+T}.
\end{align*}
   As the above inequality is valid when $k=1$, it suffices to show that
\begin{align*}
    S_{k+1}-S_k \geq (k+1)^{r+T}-k^{r+T}.
\end{align*}
    The above inequality simplifies to be
\begin{align*}
     \sum^{T-1}_{i=1}\binom {r+T}{r+i} k^{r+i}+\binom {r+T}{r}(k+1)^{r}
     \geq (k+1)^{r+T}-k^{r+T}
     =\sum^{r+T-1}_{i=0}\binom{r+T}{i}k^i.
\end{align*}
   We can further recast the above inequality as
\begin{align*}
    \binom {r+T} {r}(k+1)^{r}
     \geq \sum^{r}_{i=0}\binom{r+T}{i}k^i.
\end{align*}
   As the above inequality is easily verified by first applying the binomial expansion to $(k+1)^r$ and then to compare the corresponding coefficients of $k^i, 0 \leq i \leq T$ on both sides, this implies that inequality \eqref{2.03} is valid for all $r \geq 1, t \geq 1, s=0$ and hence
   inequality \eqref{2.1} is valid for  all $k \geq 1$, $r \geq m \geq n \geq 1, r+1=m+n$.
\end{proof}


\begin{lemma}
\label{lem2}
   For integers $k \geq 1,  r \geq m \geq n \geq s, r+s=m+n$, inequality \eqref{2.1} holds when $s=2$ or when $n=4, s=3$.
\end{lemma}
\begin{proof}
    Using the notations in Lemma \ref{lem1}, we apply the binomial expansion to see that for integers $T \geq s \geq 1, r \geq s, k \geq 1$,
\begin{align*}
   & (\Delta {\bf a}(r+1,T+1))_{k}=\sum^{T-1}_{t=0}\binom {T}{t}{\bf a}(r+1,t+1)_{k}, \\
 & (\Delta {\bf b}(r+T-s+1,r+1, T+1, s+1))_{k} \\
  =&\alpha_{r+T-s+1,r+1,T+1,s+1}\sum^{T-1}_{t=s-r}\binom {r+T-s}{r+t-s}\sum^{k}_{j=1}j^s(k-j)^{r+t-s} \\
  =& \sum^{T-1}_{t=s-r}\binom {T}{t}{\bf b}(r+t-s+1,r+1, t+1, s+1)_{k}\\
  =&\sum^{T-1}_{t=s}\binom {T}{t}{\bf b}(r+t-s+1,r+1, t+1, s+1)_{k}\\
  &+\sum^{s-1}_{t=s-r}\binom {T}{t}{\bf b}(r+t-s+1,r+1, t+1, s+1)_{k}.
\end{align*}
   It follows that for integers $i \geq 1, k \geq 1$, we have (with empty sums being $0$)
\begin{align}
\label{2.08}
   & (\Delta^{(i+1)} {\bf a}(r+1,T+1))_{k}=\sum^{T-1}_{t=0}\binom {T}{t}(\Delta^{(i)}{\bf a}(r+1,t+1))_{k}, \\
 & (\Delta^{(i+1)} {\bf b}(r+T-s+1,r+1, T+1, s+1))_{k}  \nonumber \\
  =& \sum^{T-1}_{t=s}\binom {T}{t}(\Delta^{(i)} {\bf b}(r+t-s+1,r+1, t+1, s+1))_{k} \nonumber \\
  &+\sum^{s-1}_{t=s-r}\binom {T}{t}(\Delta^{(i)} {\bf b}(r+t-s+1,r+1, t+1, s+1))_{k}. \nonumber
\end{align}

   Suppose that for fixed $r \geq s, T \geq s \geq 1$, we can show that for all $k \geq 1$,
\begin{align}
\label{2.18}
  \sum^{s-1}_{t=0}\binom {T}{t}(\Delta^{(s+1)}{\bf a}(r+1,t+1))_{k}\geq \sum^{s-1}_{t=s-r}\binom {T}{t}(\Delta^{(s+1)} {\bf b}(r+t-s+1,r+1, t+1, s+1))_{k}.
\end{align}
    As we have trivially for any $k \geq 1$,
\begin{align*}
   &(\Delta^{(s+1)}{\bf a}(r+1,s+1))_{k}
  =(\Delta^{(s+1)} {\bf b}(r+1,r+1, s+1, s+1))_{k}.
\end{align*}
   It follows that if for any $s \leq t < T$, we have
\begin{align}
\label{2.07}
   (\Delta^{(s+1)} {\bf a}(r+1,t+1))_{k} \geq (\Delta^{(s+1)} {\bf b}(r+t-s+1,r+1, t+1, s+1))_{k}, \quad k \geq 1.
\end{align}
    Then inequalities \eqref{2.08} and \eqref{2.18} imply that
\begin{align*}
   (\Delta^{(s+2)} {\bf a}(r+1,T+1))_{k} \geq (\Delta^{(s+2)} {\bf b}(r+T-s+1,r+1, T+1, s+1))_{k}, \quad k \geq 1,
\end{align*}
   which in turn implies that inequality \eqref{2.07} is valid with $t=T$ there, provided we show that inequality \eqref{2.04} is valid for $t=T, i=s+1$.

   We then conclude that inequality \eqref{2.200} is valid for all $t \geq s \geq 1$ provided that both inequalities \eqref{2.18} and \eqref{2.04} (for $1 \leq i \leq s+1$) are valid.

   To facilitate the calculation of the right-hand side expression in \eqref{2.18}, we note that 
\begin{align*}
 \sum^{s-1}_{t=s-r}\binom {T}{t}{\bf b}(r+t-s+1,r+1, t+1, s+1)_{k}= \alpha_{r+T-s+1,r+1,T+1,s+1}\sum^{r-1}_{t=0}\binom {r+T-s}{t}{\bf c}(t, s)_k, 
\end{align*}   
   where for integers $k, s \geq 1$, $t \geq 0$, 
\begin{align*}
    {\bf c}(t,s)_{k}=\sum^{k-1}_{j=0}j^{t}(k-j)^{s}.
\end{align*}

    It follows that
\begin{align*}
  & \sum^{s-1}_{t=s-r}\binom {T}{t}(\Delta^{(s+1)}{\bf b}(r+t-s+1,r+1, t+1, s+1))_{k}  \\
 = & \alpha_{r+T-s+1,r+1,T+1,s+1}\sum^{r-1}_{t=0}\binom {r+T-s}{t}(\Delta^{(s+1)}{\bf c}(t, s))_k,
\end{align*}    
     
   Now, the case $s=2$ of inequality \eqref{2.1} corresponds to the case $s=1$ of inequality \eqref{2.200}.
In this case, it is readily checked that inequality \eqref{2.18} becomes the following inequality, which is seen to be valid by comparing the corresponding coefficients of $(k+1)^t, 0 \leq t \leq r-1$:
\begin{align*}
   (k+2)^r-(k+1)^r =\sum^{r-1}_{t=0}\binom {r}{t}(k+1)^t
   \geq \alpha_{r+T,r+1,T+1,2}\sum^{r-1}_{t=0}\binom {r+T-1}{t}(k+1)^t.
\end{align*}

    It remains to prove inequality \eqref{2.04} for $t \geq s=1$, $1 \leq i \leq 2$. The case $i=1$ holds trivially and when $i=2$, inequality \eqref{2.04} becomes
\begin{align*}
    D(r,t):=2^r+2^t-2 -\frac {r!t!}{(r+t-1)!}\cdot 2^{r+t-1} \geq 0.
\end{align*}
    We may assume that $r \geq t$ and note that the above inequality becomes an identity when $t=1$.  We check that
\begin{align*}
    D(r, t+1)-D(r,t)=2^t-\frac {r!t!}{(r+t)!}\cdot 2^{r+t-1} \cdot (t+2-r).
\end{align*}
   The right-hand side expression above is easily seen to be positive when $r \geq t+2$, while $D(t, t+1)-D(t,t) \geq 0$ and $D(t+1, t+1)-D(t+1,t) \geq 0$ are consequences of the following easily verified inequalities:
\begin{align*}
   \frac {(2t+1)!}{(t+1)!t!} \geq \frac {(2t)!}{t!t!} \geq 2^t.
\end{align*}
   This shows that inequality \eqref{2.1} is valid for $s=2$.

    Next, the case $s=3, n=4$ of inequality \eqref{2.1} corresponds to the case $s=2, t=3$ of inequality \eqref{2.200}.  As inequality \eqref{2.200} holds trivially for $s=t=2$, by our discussions above, it remains to prove \eqref{2.18} for $s=2, T=3$ and \eqref{2.04} for $s=2, t=3, 1 \leq i \leq 3$.

    Inequality \eqref{2.04} for $s=2, t=3, 1 \leq i \leq 3$ are readily checked to be valid, while inequality \eqref{2.18} for $s=2, T=3$ becomes
\begin{align*}
   & (k+3)^r-2(k+2)^r+(k+1)^r+T((k+2)^r-(k+1)^r)  \\
   \geq & \sum^{r-1}_{t=0} \frac {r!T!}{2\cdot t! (r+T-2-t)!}((k+1)^t+(k+2)^t).
\end{align*}

   When $T=3$, the above inequality becomes
\begin{align*}
   (k+3)^r+4(k+2)^r+(k+1)^r \geq \frac 3{r+1}((k+3)^{r+1}-(k+1)^{r+1}).
\end{align*}
   Applying the binomial expansion to write both sides above as sums of powers of $k+1$ and by comparing the corresponding coefficients, we see that
   it suffices to show for $0 \leq i \leq r-1$,
\begin{align*}
   \binom {r}{i}(2^{r-i}+4) \geq \frac 3{r+1}\binom {r+1}{i} 2^{r+1-i},
\end{align*}
  which is equivalent to
\begin{align*}
  4(r+1-i) \geq (5+i-r)2^{r-i}.
\end{align*}
   One checks easily that the above inequality is valid for $r-4 \leq i \leq r-1$ and holds trivially when $i \leq r-5$. This proves \eqref{2.18} for $s=2, T=3$
   and hence completes the proof for the case $s=3, n=4$ of inequality \eqref{2.1}.


\end{proof}
\section{Proof of Theorem \ref{thm1}}
\label{sec 2} \setcounter{equation}{0}
   We note the following expressions for $\psi_q(x)$, which can be found in \cite[(1.3)-(1.4)]{Ba1}:
\begin{align}
\label{2.2}
   \psi_q(x)  &=   -\ln (q-1) + \ln q \left( x -\frac 1{2}- \sum^{\infty}_{n=1}\frac {q^{-nx}}{1-q^{-n}} \right ), \quad q> 1, x>0, \\
  \psi_q(x)  &=   -\ln (1-q) + \ln q \sum^{\infty}_{n=1}\frac {q^{nx}}{1-q^n}, \quad 0 < q< 1, x>0. \nonumber
\end{align}

   For two variables $x,y$, we define $\delta_{xy}=1$ when $x=y$ and $\delta_{xy}=0$ otherwise. It follows from the above expressions that
\begin{align}
\label{3.2}
  \psi_q(x) &=(x-\frac 3{2})\ln q+\psi_{1/q}(x), \\
  \psi^{(m)}_q(x) &=\delta_{1m}\ln q+\psi^{(m)}_{1/q}(x), \quad m \geq 1. \nonumber
\end{align}

  We then deduce that $G_m(x;q,c)=G_m(x;1/q,c)$, hence it suffices to prove
   Theorem \ref{thm1} by assuming that $q>1$. Using \eqref{2.2}, we see that when $q>1$,
\begin{align*}
   G_m(x;q,c)=\frac {(\ln q)^{m+1}}{c^2} \sum^{\infty}_{n=2}q^{-nx}H(m,n,q,c)+(1-\delta_{1m})(\ln q)^{m+1}q^{-x}\frac {(m-2)(1-q^{-c})}{(1-q^{-1})c},
\end{align*}
   where
\begin{align*}
  &H(m,n,q,c) \\
  =& \sum^{n-1}_{j=1}\frac {(1-q^{-jc})(1-q^{-(n-j)c})m(n-j)^{m-1}}{(1-q^{-j})(1-q^{-(n-j)})}-\frac {c(1-q^{-nc})(n^m-(m-1)n^{m-1}-\delta_{1m})}{1-q^{-n}}.
\end{align*}

   It suffices to show $H(m,n,q,c) \geq 0$ for all $n \geq 2$. We first show that for any $1 \leq j \leq n-1$, $0<c<1$,
\begin{align}
\label{2.3}
   \frac {(1-q^{-jc})(1-q^{-(n-j)c})}{(1-q^{-j})(1-q^{-(n-j)})} \geq \frac {c(1-q^{-nc})}{1-q^{-n}},
\end{align}
  with the above inequality reversed when $c>1$. We denote $y=q^{-1}$ so that $0<y<1$ and we let
\begin{align*}
   h(z;y,c)=\ln \frac {1-y^{zc}}{1-y^z}.
\end{align*}

   It is then easy to see that inequality \eqref{2.3} is equivalent to
\begin{align}
\label{2.4}
   h(j;y,c)+h(n-j;y,c) \geq \lim_{t \rightarrow 0^+}\left ( h(t;y,c)+h(n-t;y,c) \right ),
\end{align}
   with the above inequality reversed when $c>1$.

   Applying \cite[Lemma 2.4]{Gao1}, we see that \eqref{2.4} follows if we can show $h(z;y,c)$ is a concave function of $z$ when $0<c<1$ and a convex function of $z$ when $c>1$. Direct calculation shows that
\begin{align*}
   h''(z;y,c)=\frac {(\ln y)^2y^z}{(1-y^{z})^2}-\frac {(\ln y^c)^2y^{zc}}{(1-y^{zc})^2}.
\end{align*}

   As it is easy to check that the function
\begin{align*}
   x \mapsto \frac {u(\ln u)^2}{(1-u)^2}
\end{align*}
    is an increasing function of $0<u<1$, it follows readily that $h''(z;y,c) \leq 0$ when $0< c< 1$ and $h''(z;y,c) \geq 0$ when $c> 1$ , so that inequality \eqref{2.3} follows.

   Now, applying inequality \eqref{2.3} in the expression of $H(m,n,q,c)$, we see that the assertion for Theorem \ref{thm1} is valid  as long as we can show for
   integers $n \geq 2, m \geq 1$,
\begin{align}
\label{2.5}
  m\sum^{n-1}_{j=1}(n-j)^{m-1}\geq n^{m}-(m-1)n^{m-1}-\delta_{1m},
\end{align}
  with equality holds when $m=1,2$.

  It is easy to see that inequality \eqref{2.5} becomes equality when $m=1,2$  and this proves the assertion of Theorem \ref{thm1} for the case $m=1,2$. In fact, as one checks easily $-G_2(x;q,c)=G'_1(x;q,c)$, the assertion of Theorem \ref{thm1} for the case $m=2$ follows from its assertion for the case $m=1$.

  Now, we assume $m \geq 2$ and we let $r=m-1$ to recast inequality \eqref{2.5} as
\begin{align}
\label{3.4}
   (r+1)\sum^{n}_{i=1}i^r \geq (n+1)^{r+1}-r(n+1)^r, \quad r \geq 1, \quad n \geq 1.
\end{align}
   Note that by Hadamard's inequality (\cite[Lemma 2.5]{Gao1}), we have
\begin{align*}
   \int^{n+1}_nx^rdr \leq \frac {n^r+(n+1)^r}{2} \leq \frac r{r+1}(n+1)^r+\frac {1}{r+1}n^r, \quad n \geq 1.
\end{align*}
    It it easy to see that inequality \eqref{3.4} follows from the above inequality and induction (with the case $n=1$ in \eqref{3.4} following from
    the case $n=1$ of the above inequality). This completes the proof for Theorem \ref{thm1}.

\section{Proof of Theorem \ref{thm2}}
\label{sec 3} \setcounter{equation}{0}
    It follows from \eqref{3.2} that when $s \geq 2$, we have $F_{r,m,n,s}(x; q, c)=F_{r,m,n,s}(x; 1/q, c)$. Hence, it suffices to
    prove the assertion of Theorem \ref{thm2} by assuming that $0<q<1$. We start by considering the function $F_{r,m,n,s}(x; q, c)$ with $r>m \geq n>s \geq 1, r+s=m+n$. Using \eqref{2.2}, we see that when $0<q<1, s >0$,
\begin{align*}
   & F_{r, m,n,s}(x; q, c) \\
   =& \frac {(-\ln q)^{m+n}}{c^2} \sum^{\infty}_{k=2}q^{kx}\sum^{k-1}_{j=1}\frac {(1-q^{jc})(1-q^{(k-j)c})}{(1-q^j)(1-q^{k-j})}(j^{m-1}(k-j)^{n-1}-\alpha_{r,m,n,s}j^{r-1}(k-j)^{s-1}),
\end{align*}
    We then deduce that in order for $F_{r, m,n,s}(x; q, c)$ to be completely monotonic on $(0,\infty)$, it suffices to show the inner sum in the above expression is non-negative for all $k \geq 2$. We recast the inner sum above as
\begin{align*}
    \sum^{[k/2]}_{j=1}\frac {(1-q^{jc})(1-q^{(k-j)c})}{(1-q^j)(1-q^{k-j})}T_{m,n,r,s}(j;k)(1-\frac {\delta_{j\frac k{2}}}{2}),
\end{align*}
   where
\begin{align*}
    T_{m,n,r,s}(j;k)& =j^{m-1}(k-j)^{n-1}+(k-j)^{m-1}j^{n-1}-\alpha_{r,m,n,s}(j^{r-1}(k-j)^{s-1}+(k-j)^{r-1}j^{s-1}) \\
                    &=j^{r-1}(k-j)^{s-1}a(k/j-1;r-s, m-s, \alpha_{r,m,n,s}),
\end{align*}
   where for integers $m > n \geq 1$, any fixed constant $0 < c < 1$, the function $t \mapsto a(t;m,n,c)$ is defined as in \cite[Lemma 2.7]{Gao1} (we note here
   that it is easy to verify that $0<\alpha_{r,m,n,s}<1$).

     It is shown in \cite[Lemma 2.7]{Gao1} that $a(t;m,n,c)$ has exactly one root when $t \geq 1$ for any integer $m > n \geq 1$, any fixed constant $0 < c < 1$. Note that $\lim_{t \rightarrow \infty}a(t;m,n,c)=-\infty, a(1;m,n,c)>0$, it follows
    that there exists an integer $1 \leq j_0 \leq [k/2]$ such that $a(k/j-1;r-s, m-s, \alpha_{r,m,n,s})<0$ for $j < j_0$ and $a(k/j-1;r-s, m-s, \alpha_{r,m,n,s}) \geq 0$ for $j_0 \leq j \leq [k/2]$.

      If $j_0=1$, then the assertion of the theorem holds trivially. Otherwise, note that our argument for inequality \eqref{2.3} shows that the sequence
\begin{align*}
   \Big \{ \frac {(1-q^{jc})(1-q^{(k-j)c})}{(1-q^j)(1-q^{k-j})} \Big \}^{[k/2]}_{j=1}
\end{align*}
    is a positive increasing sequence, thus we have
\begin{align*}
   \frac {(1-q^{c})(1-q^{(k-1)c})}{(1-q)(1-q^{k-1})}T_{m,n,r,s}(1;k) \geq \frac {(1-q^{2c})(1-q^{(k-2)c})}{(1-q^2)(1-q^{k-2})}T_{m,n,r,s}(1;k),
\end{align*}
    so that
\begin{align*}
   \sum^{2}_{j=1}\frac {(1-q^{jc})(1-q^{(n-j)c})}{(1-q^j)(1-q^{n-j})}T_{m,n,r,s}(j;k) \geq \frac {(1-q^{2c})(1-q^{(n-2)c})}{(1-q^2)(1-q^{n-2})}\sum^{2}_{j=1}T_{m,n,r,s}(j;k).
\end{align*}
     Consider the sequence of the partial sums:
\begin{align*}
   \Big \{ \sum^i_{j=1}T_{m,n,r,s}(j;k)\Big \}^{[k/2]-1}_{i=1}.
\end{align*}

    If there exists an integer $1 \leq i_0 \leq [k/2]-1$ such that the terms in the above sequence are negative when $i < i_0$ and the term corresponding to $i=i_0$ is non-negative, then we can repeat the above process to see that (note that we must have $i_0 \geq j_0$ here)
\begin{align*}
    & \sum^{[k/2]}_{j=1}\frac {(1-q^{jc})(1-q^{(n-j)c})}{(1-q^j)(1-q^{n-j})}T_{m,n,r,s}(j;k)(1-\frac {\delta_{j\frac k{2}}}{2}) \\
    \geq & \frac {(1-q^{i_0 c})(1-q^{(n-j)i_0})}{(1-q^{i_0})(1-q^{n-i_0})}\sum^{i_0}_{j=1}T_{m,n,r,s}(j;k)(1-\frac {\delta_{j\frac k{2}}}{2}) \\
    & +\sum^{[k/2]}_{j=i_0+1}\frac {(1-q^{jc})(1-q^{(n-j)c})}{(1-q^j)(1-q^{n-j})}T_{m,n,r,s}(j;k)(1-\frac {\delta_{j\frac k{2}}}{2}) \geq 0.
\end{align*}
    If no such $i_0$ exists, then we have
\begin{align*}
    & \sum^{[k/2]}_{j=1}\frac {(1-q^{jc})(1-q^{(n-j)c})}{(1-q^j)(1-q^{n-j})}T_{m,n,r,s}(j;k)(1-\frac {\delta_{j\frac k{2}}}{2}) \\
    \geq & \frac {(1-q^{i_0 c})(1-q^{(n-j)[k/2]})}{(1-q^{[k/2]})(1-q^{n-[k/2]})}\sum^{[k/2]}_{j=1}T_{m,n,r,s}(j;k)(1-\frac {\delta_{j\frac k{2}}}{2}).
\end{align*}

    Hence, it remains to show that
\begin{align*}
    \sum^{[k/2]}_{j=1}T_{m,n,r,s}(j;k)(1-\frac {\delta_{j\frac k{2}}}{2}) \geq 0.
\end{align*}
    We can recast the above inequality as inequality \eqref{2.1} and the assertion of Theorem \ref{thm2} now follows from Lemmas \ref{lem1}-\ref{lem2}.



\end{document}